\documentclass[11pt,reqno]{amsart}
\usepackage{fullpage,times,graphicx,amssymb,amsmath,psfrag,xcolor,natbib}
\usepackage[colorlinks,citecolor=bbluegray,linkcolor=ddarkbrown,urlcolor=blue,breaklinks]{hyperref}
\usepackage{algorithmic,algorithm}

\usepackage{auto-pst-pdf} 

\definecolor{ddarkbrown}{rgb}{0.5,0.2,0.05} \definecolor{bbluegray}{rgb}{0.05,0,0.5}

\newtheorem{theorem}{Theorem}[section]
\newtheorem{proposition}[theorem]{Proposition}

\newtheorem{lemma}[theorem]{Lemma}
\newtheorem{corollary}[theorem]{Corollary}
\renewenvironment{proof}{\textbf{Proof.}}{\QED\bigskip}

\newcommand{\BEAS}{\begin{eqnarray*}}
\newcommand{\EEAS}{\end{eqnarray*}}
\newcommand{\BEA}{\begin{eqnarray}}
\newcommand{\EEA}{\end{eqnarray}}
\newcommand{\BEQ}{\begin{equation}}
\newcommand{\EEQ}{\end{equation}}
\newcommand{\BIT}{\begin{itemize}}
\newcommand{\EIT}{\end{itemize}}
\newcommand{\BNUM}{\begin{enumerate}}
\newcommand{\ENUM}{\end{enumerate}}

\newcommand{\BA}{\begin{array}}
\newcommand{\EA}{\end{array}}

\newcommand{\ones}{\mathbf 1}

\newcommand{\reals}{{\mathbb R}}

\newcommand{\symm}{{\mbox{\bf S}}}  

\newcommand{\diam}{\mathop{\bf diam}}

\newcommand{\Rank}{\mathop{\bf Rank}}

\newcommand{\Card}{\mathop{\bf Card}}
\newcommand{\Tr}{\mathop{\bf Tr}}
\newcommand{\diag}{\mathop{\bf diag}}
\newcommand{\lambdamax}{{\lambda_{\rm max}}}

\newcommand{\idm}{\mathbf{I}}

\newcommand{\Expect}{\textstyle\mathop{\bf E}}

\newcommand{\Prob}{\mathop{\bf Prob}}

\newcommand{\QED}{~~\rule[-1pt]{6pt}{6pt}}

\newcommand{\dsp}{\displaystyle}

\newcommand{\Sht}{\hat \Sigma}

\begin{document}
\title{Approximation Bounds for Sparse Principal Component Analysis}
\author{Alexandre d'Aspremont}
\address{In alphabetical order.\vskip 1ex CMAP, Ecole Polytechnique, UMR CNRS 7641, Palaiseau, France.}
\email{alexandre.daspremont@m4x.org}

\author{Francis Bach}
\address{INRIA, SIERRA Project-Team,\vskip 0ex Laboratoire d'Informatique de l'\'Ecole Normale Sup\'erieure, Paris, France.}
\email{francis.bach@ens.fr}

\author{Laurent El Ghaoui}
\address{EECS  Department, U.C. Berkeley, Berkeley, CA 94720, USA.}
\email{elghaoui@eecs.berkeley.edu}

\date{\today}

\begin{abstract}
We produce approximation bounds on a semidefinite programming relaxation for sparse principal component analysis. These bounds control approximation ratios for tractable statistics in hypothesis testing problems where data points are sampled from Gaussian models with a single sparse leading component. 
\end{abstract}
\maketitle

We study approximation bounds for a semidefinite relaxation of the sparse eigenvalue problem, written here in penalized form
\[
\max_{\|x\|_2=1} ~ x^T\Sigma x - \rho \Card(x)
\]
in the variable $x\in\reals^n$, where $\Sigma\in\symm_n$ and $\rho\geq 0$. Sparse eigenvalues appear in many applications in statistics and machine learning. Sparse eigenvectors are often used, for example, to improve the interpretability of principal component analysis, while sparse eigenvalues control recovery thresholds in compressed sensing \citep{Cand07}. Several convex relaxations and greedy algorithms have been developed to find approximate solutions (see  \citet{dasp04a,dAsp08b,Jour08,Jour08a} among others), but except in simple scenarios where $\rho$ is  small and the two leading eigenvalues of $\Sigma$ are separated, very little is known about the tightness of these approximation methods.

Here, using randomization techniques based on \citep{Ben-02}, we derive simple approximation bounds for the semidefinite relaxation derived in \citep*{dAsp08b}. We do not produce a constant approximation ratio and our bounds depend on the optimum value of the semidefinite relaxation: the higher this value, the better the approximation. A similar behavior was observed by \citet{Zwic99} for the semidefinite relaxation to MAXCUT, who showed that the classical approximation ratio of \citet{Goem95} can be improved when the value of the cut is high enough.

We then show that, in some applications, it is possible to bound a priori the optimum value of the semidefinite relaxation, hence produce a lower bound on the approximation ratio. In particular, following recent works by \citep{Amin08,Bert12}, we focus on the problem of detecting the presence of a (significant) sparse principal component in a Gaussian model, hence test the significance of eigenvalues isolated by sparse principal component analysis. More precisely, we apply our approximation results to the problem of discriminating between the two Gaussian models
\[
\mathcal{N}\left(0,\idm_n\right)
\quad \mbox{and} \quad
\mathcal{N}\left(0,\idm_n+\theta vv^T\right)
\]
where $v\in\reals^n$ is a sparse vector with unit Euclidean norm and cardinality $k$. 
We use a convex relaxation for the sparse eigenvalue problem to produce a tractable statistic for this hypothesis testing problem and show that in a high-dimensional setting where the dimension $n$, the number of samples $m$ and the cardinality $k$ grow towards infinity proportionally, the detection threshold on $\theta$ remains finite.

More broadly speaking, in the spirit of smoothed analysis \citep{Spie01}, this shows that analyzing the performance of semidefinite relaxations on random problem instances is sometimes easier and provides a somewhat more realistic description of typical approximation ratios. Another classical example of this phenomenon is a MAXCUT-like problem arising in statistical physics, for which explicit (asymptotic) formulas can be derived for certain random instances, e.g. the Parisi formula \citep{Meza87,Meza09,Tala10} for computing the ground state of spin glasses in the Sherrington-Kirkpatrick model. It thus seems that comparing the performance of convex relaxations on random problem instances (e.g. in detection problems) often yields a more nuanced understanding of their performance in cases where uniform approximation ratios are hard to derive or analyze.

The paper is organized as follows. The next section recalls a few definitions on sparse eigenvalue problems. Section~\ref{s:relax} recalls the construction of the semidefinite relaxation in \citep{dAsp08b}. Section~\ref{s:approx} derives approximation bounds on the solution of this relaxation. Section~\ref{s:detect} studies the performance of these approximation bounds on the sparse eigenvector detection problem. Section~\ref{s:algos} presents some algorithms for solving the semidefinite relaxation used as a test statistic in the detection problem. Finally, we present some numerical results in Section~\ref{s:numexp}.

\section{Sparse eigenvalues}
We begin by formally defining sparse eigenvalues. Let $\Sigma\in\symm_n$ be a symmetric matrix. We define the sparse maximum eigenvalues of the matrix $\Sigma$ as
\BEQ\label{eq:sp-max-eig}
\BA{rll}
\lambda_\mathrm{max}^k(\Sigma)\triangleq &\mbox{max.} & x^T \Sigma x\\
&\mbox{s.t.} & \Card(x)\leq k\\
&& \|x\|_2=1,
\EA\EEQ
in the variable $x\in\reals^n$ where the parameter $k>0$ controls the sparsity of the solution. We can similarly define sparse minimum eigenvalues as
\BEQ\label{eq:sp-min-eig}
\BA{rll}
\lambda_\mathrm{min}^k(\Sigma)\triangleq &\mbox{min.} & x^T \Sigma x\\
&\mbox{s.t.} & \Card(x)\leq k\\
&& \|x\|_2=1,
\EA\EEQ
in the variable $x\in\reals^n$. Because $\lambda_\mathrm{max}^k(\Sigma+ \alpha \idm)$ is affine in $\alpha$, we have
\[
\lambda_\mathrm{min}^k(\Sigma) = \lambda_\mathrm{max}(\Sigma)-\lambda_\mathrm{max}^k\left(\lambda_\mathrm{max}(\Sigma)\idm-\Sigma\right)
\]
and the following sections will be focused on approximating $\lambda_\mathrm{max}^k(\Sigma)$.

\section{Semidefinite relaxation} \label{s:relax}
Here, we first recall the semidefinite relaxation for~\eqref{eq:sp-max-eig} derived in~\citep*{dAsp08b}. We assume without loss of generality that $\Sigma\in\symm_n$ is positive semidefinite (we can always add a multiple of the identity) and that the $n$ variables are ordered by decreasing marginal variances, i.e. that $\Sigma_{11} \geq \ldots \geq \Sigma_{nn}$. We also assume that we are given a square root $A$ of the matrix $\Sigma$ with 
\[
\Sigma = A^T A,
\] 
where $A\in\reals^{n \times n}$ and we denote by $a_1,\dots,a_n \in \reals^n$ the columns of $A$. Note that the problem and our algorithms are invariant by permutations of $\Sigma$ and by the choice of square root $A$. In practice, we are very often given the data matrix $A$ instead of the covariance $\Sigma$. As we will see below, we can directly exclude variables for which $\Sigma_{ii}<\rho$, hence we can assume w.l.o.g. that 
\[
0 < \rho < \min_{i\in[1,n]} \Sigma_{ii}.
\] 
If this condition is not satisfied, the variable $i$ will never be part of the optimal support and we can focus on the reduced problem.

\subsection{Relaxation bounds on sparse eigenvalues}
We can rewrite the maximum eigenvalue problem in terms of the data matrix $A$. We start by writing
\BEAS
\lambda_\mathrm{max}^k(\Sigma) &=& \max_{\substack{\Card(x)\leq k\\ \|x\|_2=1}}~ x^T \Sigma x\\
&=& \max_{\substack{u\in\{0,1\}^n\\ \ones^Tu=k}} ~ \lambda_\mathrm{max}\left(\sum_{i=1}^n u_i a_ia_i^T\right)\\
&=& \max_{\substack{u\in\{0,1\}^n\\ \ones^Tu=k}} ~ \max_{\|x\|_2=1} ~ \sum_{i=1}^n u_i (a_i^Tx)^2\\
&=& \max_{\|x\|_2=1} ~ \max_{\substack{u\in\{0,1\}^n\\ \ones^Tu=k}} ~ \sum_{i=1}^n u_i (a_i^Tx)^2,
\EEAS
and use the fact that
\[
\max_{\substack{u\in\{0,1\}^n\\ \ones^Tu=k}} ~ \sum_{i=1}^n u_i b_i ~=~ \min_{\rho \geq 0} ~\left\{ \sum_{i=1}^n (b_i-\rho)_+ + \rho k \right\}
\]
for any $b\in\reals^n$, to write
\BEAS
&& \max_{\|x\|_2=1} ~ \max_{\substack{u\in\{0,1\}^n\\ \ones^Tu=k}} ~ \sum_{i=1}^n u_i (a_i^Tx)^2\\
&= & \max_{\|x\|_2=1} ~ \min_{\rho \geq 0} ~\left\{ \sum_{i=1}^n \left((a_i^Tx)^2-\rho\right)_+ + \rho k\right\}\\
&=& \max_{\substack{\Rank(X)=1\\X\succeq 0, \Tr(X)=1}} ~ \min_{\rho \geq 0} ~\left\{ \sum_{i=1}^n \left(X^{1/2}(a_ia_i^T-\rho\idm)X^{1/2}\right)_+ + \rho k\right\},
\EEAS
where the last equality follows from the fact, when $\Rank(X)=1$ the only potentially nonnegative eigenvalue of $\left(X^{1/2}(a_ia_i^T-\rho\idm)X^{1/2}\right)$ is $(a_i^Tx)^2-\rho$, and $X^{1/2}=X=xx^T$ when $\|x\|_2=1$. We then produce a semidefinite relaxation for~\eqref{eq:sp-max-eig} by simply dropping the rank constraint to get the following bound
\BEQ\label{eq:sdp-max-eig}
\lambda_\mathrm{max}^k(\Sigma) ~\leq ~\min_{\rho \geq 0} ~\left\{ \max_{\substack{\Tr(X)=1\\X\succeq 0}} ~ \sum_{i=1}^n \left(X^{1/2}(a_ia_i^T-\rho\idm)X^{1/2}\right)_+ + \rho k\right\} 
\EEQ
which is equivalent to a semidefinite program. Note that because $\Rank(X)=1$ defines a non-convex set, we cannot simply switch the $\min$ and the $\max$ and this last inequality is potentially strict. i.e. the semidefinite relaxation only produces an upper bound on $\lambda_\mathrm{max}^k(\Sigma)$.

\subsection{Penalized problem} We now focus on the inner optimization problem in~\eqref{eq:sdp-max-eig}. Starting from a penalized version of problem~\eqref{eq:sp-max-eig}, written
\BEQ\label{eq:sp-pen-eig}
\phi(\rho) \triangleq \max_{\|x\|_2=1} ~ x^T\Sigma x - \rho \Card(x)
\EEQ
it was shown in \cite{dAsp08b} that
\BEAS
\phi(\rho)&=& \max_{\|x\|_2=1} ~ \sum_{i=1}^n\left((a_i^Tx)^2-\rho\right)_+\\
&=& \max_{\substack{\Rank(X)=1\\X\succeq 0, \Tr(X)=1}} ~  \sum_{i=1}^n\Tr \left(X^{1/2}(a_ia_i^T-\rho\idm)X^{1/2}\right)_+
\EEAS
and we write $\psi(\rho)$ the semidefinite relaxation of this last problem
\BEQ \label{eq:psi}
\BA{rll}
\psi(\rho) \triangleq &\mbox{max.}&  \sum_{i=1}^n\Tr(X^{1/2}a_ia_i^TX^{1/2}-\rho X)_+ \\
&\mbox{s.t.} & \Tr(X)=1,~X\succeq 0,
\EA\EEQ
which is equivalent to a semidefinite program \citep{dAsp08b} and is the inner problem in~\eqref{eq:sdp-max-eig}. In the next section, we use this quantity as a test statistic for detecting significant sparse eigenvectors.

\section{Approximation Bounds}\label{s:approx}
Using the randomization argument detailed in \citep{Ben-02,El-G06}, we can derive an explicit bound on the quality of the semidefinite relaxation~\eqref{eq:psi}.
\begin{proposition}\label{eq:prop-approx}
Let us call $X$ the optimal solution to problem~\eqref{eq:psi} and let $r=\Rank(X)$, we have 
\BEQ\label{eq:rand-bounds}
n\rho~ \vartheta_r\left(\frac{\psi(\rho)}{n\rho}\right) \leq\phi(\rho)\leq \psi(\rho),
\EEQ
where
\BEQ\label{eq:theta-r}
\vartheta_r(x) \triangleq \Expect \left[\left(x \xi_1^2 - \frac{1}{r-1} \sum_{j=2}^r  \xi_j^2 \right)_+\right]
\EEQ
controls the approximation ratio.
\end{proposition}
\begin{proof}
We can assume w.l.o.g. that $0 < \rho < \min_{i\in[1,n]} \Sigma_{ii}$. 
This means that $B_i(X)=X^{1/2}(a_ia_i^T-\rho\idm)X^{1/2}$ has rank $r$ and exactly one positive eigenvalue $\alpha_i$, with $\alpha_i=\Tr B_i(X)_+$ for $i=1,\ldots,n$. We also denote by $-\beta^i_j$ for $j=2,\ldots,k$, the $(k-1)$ negative eigenvalues of $B_i(X)$. We follow the randomization procedure in \citep{Ben-02,El-G06} and let $\xi$ denote normally distributed variables on $\reals^n$, we have, using the rotational invariance of the normal distribution
\[
\Expect \left[\left( \xi^T B_i(X) \xi\right)_+\right] = \Expect \left[\left(\alpha_i\xi_1^2 - \sum_{j=2}^r \beta^i_j \xi_j^2 \right)_+\right], \quad \mbox{for }i=1,\ldots,n.
\]
We then get
\[
\sum_{j=2}^r \beta_j^i = \Tr(B(X))_+ -\Tr(B(X))= \alpha_i - (a_i^TXa_i-\rho) \leq \rho
\] 
because $\lambdamax(B_i(X))\leq a_i^TXa_i$, hence
\BEAS
\Expect \displaystyle \left[(\xi^TB_i(X)\xi)_+ \right] 
& \geq & \min_\beta \left\{\Expect \left[ \left(\alpha_i\xi_1^2 - \sum_{j=2}^r \beta^i_j \xi_j^2 \right)_+\right] :~ \sum_{j=2}^r \beta_j^i \leq \rho,~ \beta_j^i \geq 0 \right\}\\
& = & \Expect \left[\left(\alpha_i\xi_1^2 - \frac{\rho}{r-1} \sum_{j=2}^r  \xi_j^2 \right)_+\right]
\EEAS
by convexity and symmetry. By homogeneity and convexity, with $\psi(\rho) =\sum_{i=1}^n\alpha_i$, we then get
\BEAS
\Expect \displaystyle \left[\sum_{i=1}^n (\xi^TB_i(X)\xi)_+\right] & \geq & \sum_{i=1}^n \Expect \left[\left(\alpha_i\xi_1^2 - \frac{\rho}{r-1} \sum_{j=2}^r  \xi_j^2 \right)_+\right]\\
& \geq & \Expect \left[\left(\psi(\rho) \xi_1^2 -\frac{n\rho}{r-1} \sum_{j=2}^r  \xi_j^2 \right)_+\right],
\EEAS
and we define $\vartheta_r(x)$ as in\eqref{eq:theta-r} above. Having shown
\[
\Expect \left[\sum_{i=1}^n (\xi^TB_i(X)\xi)_+ \right] \geq n\rho~\vartheta_r\left(\frac{\psi(\rho)}{n\rho}\right),
\]
and using $\Expect[\xi^TX\xi]=\Tr(X)=1$, we get
\[
\Expect \left[\sum_{i=1}^n (\xi^TB_i(X)\xi)_+ \right]  \geq n\rho~ \vartheta_r\left(\frac{\psi(\rho)}{n\rho}\right)~\Expect[\xi^TX\xi],
\]
and this bound implies that there exists a nonzero $\xi$ such that
\[
\sum_{i=1}^n (\xi^TB_i(X)\xi)_+ \geq n\rho~ \vartheta_r\left(\frac{\psi(\rho)}{n\rho}\right)~ (\xi^TX\xi).
\]
Suppose we set
\[v_i=\left\{\BA{ll}
1 & \mbox{if } (\xi^TB_i(X)\xi) >0\\
0 & \mbox{otherwise, }\\
\EA\right. 
\]
we now have
\[
\xi^T \left( \sum_{i=1}^n v_iB_i(X) \right) \xi \geq n\rho~ \vartheta_r\left(\frac{\psi(\rho)}{n\rho}\right)~(\xi^TX\xi),
\]
which is also, with $z=X^{1/2}\xi$ and $B_i(X)=XB_iX$
\[
z^T \left( \sum_{i=1}^n v_iB_i \right) z \geq n\rho~ \vartheta_r\left(\frac{\psi(\rho)}{n\rho}\right) ~ z^Tz.
\]
This finally means that for our choice of $v$, with $B_i=a_i a_i^T-\rho$
\[
\phi(\rho)=\max_{v\in\{0,1\}^n}\lambdamax\left( \sum_{i=1}^n u_i a_ia_i^T\right)-\rho\Card(u)\geq\lambdamax\left( \sum_{i=1}^n v_i B_i\right) \geq n\rho~\vartheta_r\left(\frac{\psi(\rho)}{n\rho}\right),
\]
hence our lower bound $n\rho~ \vartheta_r(\psi(\rho)/n\rho) <\phi(\rho)$ (which holds whenever $X$ is feasible in~\eqref{eq:psi}, i.e. whenever $X\succeq 0$ with $\Tr X=1$). Furthermore, if $X$ is an optimal solution of the relaxation in~\eqref{eq:psi}, we also get an upper bound on $\phi(\rho)$, with
\[
n\rho~ \vartheta_r\left(\frac{\psi(\rho)}{n\rho}\right) \leq\phi(\rho)\leq \psi(\rho)
\]
which is the desired result.
\end{proof}

An explicit formula involving trigonometric integrals was derived in~\citep{El-G06} (note that our definition for the function $\vartheta(x)$ is slightly different here). When $r$ is large, we can approximate $\vartheta_r(\cdot)$ by the function
\BEQ\label{eq:vartheta}
\vartheta(x) \triangleq \Expect \left[\left(x \xi^2 - 1 \right)_+\right]
\EEQ
where $\xi\sim\mathcal{N}(0,1)$. Indeed, the central limit theorem shows that
\[
\frac{\sqrt{r}}{r-1}\sum_{j=2}^r (\xi_i^2-1)  \xrightarrow{\mathcal{L}} \mathcal{N}(0,1),
\]
when $r$ grows to infinity. By convexity, we also have $\vartheta(x) \leq \vartheta_r(x)$.
The function $\vartheta(\cdot)$ can be computed explicitly, with
\BEAS
\vartheta(x) &=& \Expect \left[\left(x \xi^2 - 1 \right)_+\right]\\
&=& 2 \int_{x^{-\frac{1}{2}}}^\infty x u^2 \frac{e^{-u^2/2}}{\sqrt{2 \pi}}du -2\mathcal{N}\left(-x^{-\frac{1}{2}}\right)\\
&=& \frac{2e^{-1/2x}}{\sqrt{2\pi x}} + 2(x-1) \mathcal{N}\left(-x^{-\frac{1}{2}}\right),
\EEAS
where $\mathcal{N}(\cdot)$ is the Gaussian cumulative density function. As with all results based on the central limit theorem, the approximation starts to be very good at relatively low values of $r$. The fact that $\vartheta'(0^+)=0$ means that we cannot obtain a constant approximation ratio (\`a la MAXCUT). However, because $\vartheta(x)$ is convex and increasing, we can still derive meaningful lower bounds in~\eqref{eq:rand-bounds} if we can bound $\psi(\rho)/n\rho$ from below by a given $c>0$, with
\BEQ\label{eq:approx-rat}
\frac{\vartheta(c)}{c} \psi(\rho) \leq n\rho~ \vartheta\left(\frac{\psi(\rho)}{n\rho}\right) \leq\phi(\rho), \quad\mbox{when } c\leq \frac{\psi(\rho)}{n\rho}.
\EEQ
We also observe that $\psi(0^+) = \lambdamax(\Sigma)>0$, and $\lim _{x\rightarrow \infty}\vartheta(x)/x = 1$ mean that, when $n$ is fixed, we have
\[
\frac{\vartheta\left({\psi(\rho)}/{n\rho}\right)}{{\psi(\rho)}/{n\rho}} \underset{\rho \rightarrow 0}{\longrightarrow} 1,
\]
i.e. the approximation ratio converges to one as $\rho$ goes to zero (and solutions get less sparse). In fact
\[
\vartheta(x)/x=1-\sqrt{\frac{2}{\pi}}\,x^{-1/2}- x^{-1} + O(x^{-3/2}),\quad \mbox{as }x\rightarrow \infty.
\]
We illustrate these last points by plotting $\vartheta(x)$ and $\vartheta(x)/x$ on the interval $[0,2]$ in Figure~\ref{fig:theta}, together with the function $\vartheta_r(x)$ for $r=5$.

\begin{figure}[h!]
\begin{center}
\begin{tabular}{ccc}
\psfrag{c}[t][b]{$x$}
\psfrag{theta}[b][t]{$\vartheta(x)$}
\psfragfig[width=.45\textwidth]{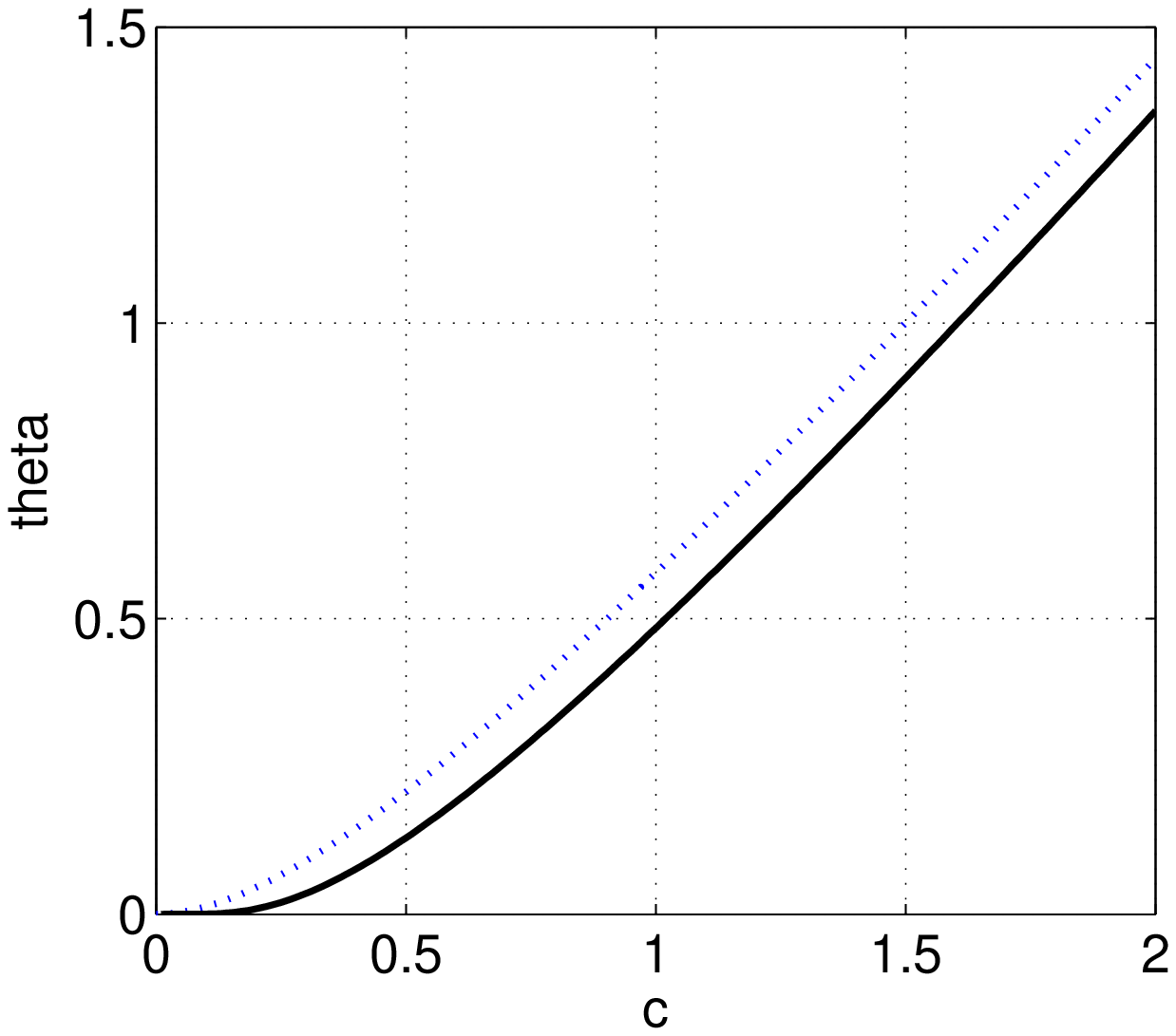}& &
\psfrag{c}[t][b]{$x$}
\psfrag{tcoverc}[b][t]{$\vartheta(x)/x$}
\psfragfig[width=.45\textwidth]{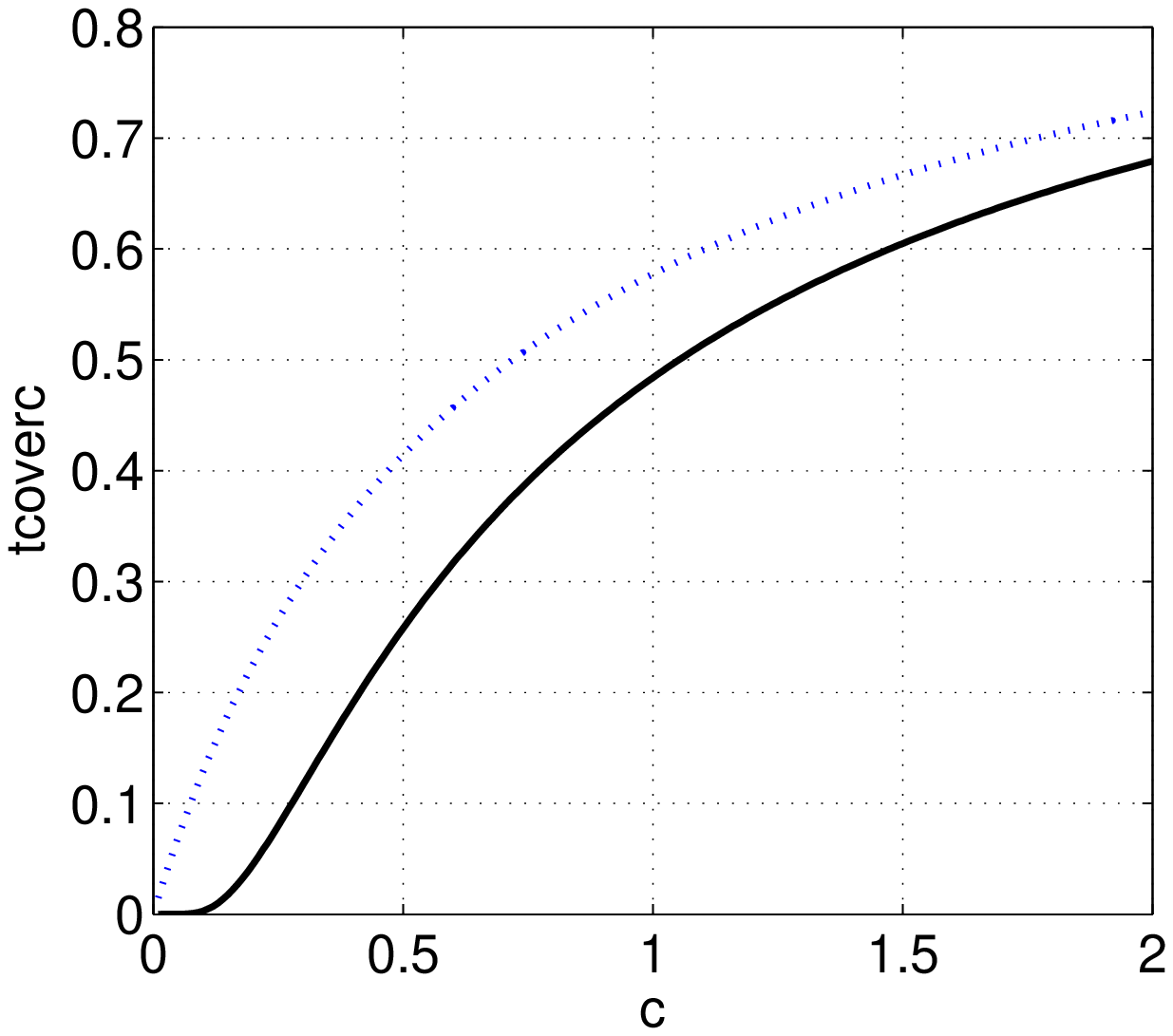}
\end{tabular}
\caption{{\em Left:} The functions $\vartheta(x)$ (solid black line), and $\vartheta_{5}(x)$ (dotted blue line).
{\em Right:} The functions $\vartheta(x)/x$ (solid black line), and $\vartheta_{5}(x)/x$ (dotted blue line). 
\label{fig:theta}}
\end{center}
\end{figure}

Remark also that a naive lower bound on $\psi(\rho)/n\rho$, hence on the approximation ratio, can be obtained by plugging $X=\idm/n$ in problem~\eqref{eq:psi}, which yields
\BEAS
\frac{\psi(\rho)}{n\rho} & \geq & \frac{\sum_{i=1}^n (\Sigma_{ii}-\rho)_+}{n^2\rho}\\
& = & \frac{\Tr(\Sigma)-\rho}{n^2\rho}
\EEAS
because we have assumed that $\rho \leq \min_i \Sigma_{ii}$. This shows for example that when 
\[
\rho \leq \frac{\Tr(\Sigma)}{n^2+1}
\]
then the approximation ration is greater than $\vartheta(1)$. We will see below that this range of values for $\rho$ is actually quite typical in detection problems where the matrix $\Sigma$ is Wishart.

\section{Detection problems} \label{s:detect}
In this section, we focus on the problem of detecting the presence of a sparse leading component in a Gaussian model. It was shown in \citep{Bert12} that the sparse eigenvalue statistic is minimax optimal in this setting. Computing sparse maximum eigenvalues is NP-Hard, but we show here that the relaxation detailed in the previous section achieve  detection rates that are a multiple of the minimax optimum, in a high-dimensional setting where the ambient dimension $n$, the number of samples $m$ and the sparsity level $k$ all grow towards infinity proportionally. More specifically, we focus on the following hypothesis testing problem, where
\BEQ\label{eq:hyp}
\left\{\BA{rl}
\mathcal{H}_0:&x \sim\mathcal{N}\left(0,\idm_n\right)\\
\mathcal{H}_1: &x \sim\mathcal{N}\left(0,\idm_n+\theta vv^T\right)
\EA\right.
\EEQ
where $\theta>0$ and $v\in\reals^n$ is a sparse vector satisfying $\Card(v)\leq k^*$ and $\|v\|_2=1$. Given $m$ sample  variables $x_i\in\reals^n$, we let $\Sht\in\symm_n$ be the sample covariance matrix, with
\[
\Sht=\frac{1}{m}\sum_{i=1}^m x_ix_i^T.
\]
We will now seek to bound the value of the statistics $\phi(\rho)$ and $\psi(\rho)$ defined in~\eqref{eq:sp-pen-eig} and~\eqref{eq:psi} respectively, under the two hypotheses above. 

\subsection{The optimal statistic $\phi(\rho)$}
We start by the easy part, namely bounding $\phi(\rho)$ from below under $\mathcal{H}_1$.

\begin{proposition}\label{prop:h1}
Given $\Sht\in\symm_n$ under $\mathcal{H}_1$, we have 
\BEQ\label{eq:lb-h1}
\phi(\rho) \geq 1 + \theta -\rho k^* -2(1+\theta)\sqrt{\frac{\log(1/\delta)}{m}}
\EEQ
with probability $1-\delta$.
\end{proposition}
\begin{proof}
By construction, we have $\phi(\rho) \geq \lambda^{k^*}_\mathrm{max}(\Sht) - \rho k^*$, and \citep[Prop.\,4.1]{Bert12} then yields the desired result.
\end{proof}

Using again the results in \citep{Bert12}, we now show an upper bound on the value of the statistic $\phi(\rho)$ under $\mathcal{H}_0$.

\begin{proposition}\label{prop:h0}
Given $\Sht\in\symm_n$ under $\mathcal{H}_0$, and assuming $\rho \geq \Delta/m$, where
\[
\Delta=4\log(9en/k^*)+4\log(1/\delta)
\] 
we have 
\BEQ\label{eq:ub-h0}
\phi(\rho) \leq 1 + \frac{4k^*\rho}{e\Delta} + \frac{1}{\rho m/\Delta -1}
\EEQ
with probability $1-2\delta$ when $\delta$ is small enough.
\end{proposition}
\begin{proof}
Under $\mathcal{H}_0$, \citep[Prop.\,4.2]{Bert12} shows
\[
\Prob\left[\lambda^{k}_\mathrm{max}(\Sht) -\rho k \geq 1 -\rho k +4 \sqrt{t/m} +4t/m \right]\leq {n \choose k} 9^k e^{-t}
\]
and writing $1+u=1 -\rho k +4 v/\sqrt{m} +4v^2/m$, with $v=\sqrt{t}$, yields
\[
v=\frac{-\sqrt{m}+\sqrt{m(1+u+\rho k)}}{2}
\]
hence
\[
\Prob\left[\lambda^{k}_\mathrm{max}(\Sht) -\rho k \geq 1 + u \right]\leq {n \choose k} 9^k e^{-v^2}.
\]
Using the fact that $\phi(\rho)=\max_k \lambda^{k}_\mathrm{max}(\Sht) -\rho k$ and ${n \choose k} \leq \left(\frac{en}{k}\right)^k$ we then get, using union bounds
\[
\Prob\left[ \phi(\rho) \geq 1 + u \right]\leq \sum_{k=1}^{n} \exp\left(k\log\frac{9en}{k}-\frac{m}{4}(\sqrt{(1+u+\rho k)}-1)^2\right).
\]
We write
\[
k\log\frac{9en}{k}-\frac{m}{4}(\sqrt{(1+u+\rho k)}-1)^2= k\log\frac{9en}{k^*}-\frac{m}{4}(\sqrt{(1+u+\rho k)}-1)^2 + (\log k^*-\log k)k.
\]
When $\alpha<1$, the function 
\[
(\sqrt{1+x}-1)^2-\alpha x
\]
is convex and reaches its minimum when $x=1/(1-\alpha)^2-1$, with value $-\alpha^2/(1-\alpha)$. A similar argument shows that $(\log k^* -\log k)k \leq k^*/e$. Setting $x=u+\rho k$ we get
\[
\alpha u + \alpha \rho k - (\sqrt{1+u+\rho k}-1)^2 \leq \alpha^2/(1-\alpha)
\]
Setting $\alpha=\Delta/\rho m$ above, imposing $\rho > \Delta/m$, and
\[
\alpha u= \alpha^2/(1-\alpha) + \frac{4k^*}{em}
\]
we can ensure 
\[
\frac{m}{4}\left(\frac{\Delta}{m}k+ (\log k^*-\log k)\frac{4k}{m} -(\sqrt{(1+u+\rho k)}-1)^2 \right) \leq 0
\]
for all $k\geq 1$, hence
\[
k\log\frac{9ep}{k^*}-\frac{m}{4}(\sqrt{(1+u+\rho k)}-1)^2 + (\log k^*-\log k)k \leq -k \log(1/\delta)
\]
and
\[
\Prob\left[ \phi(\rho) \geq 1 + u \right] \leq \frac{\delta}{1-\delta}
\]
which yields the desired result.
\end{proof}

We now use these last two results to determine the minimum signal level $\theta$ which can be detected using the statistic $\phi(\rho)$. We define the following levels
\BEQ\label{eq:levels-phi}
\left\{\BA{l}
\tau_0=1 + \sqrt{\frac{k^*(\Delta+4/e)}{m}} + \frac{4k^*}{e m}+\frac{4}{e\Delta}\sqrt{\frac{k^*\Delta}{m(1+4/(e\Delta))}} \\ 
\tau_1=1 + \theta -\sqrt{\frac{k^*\Delta}{m(1+4/(e\Delta))}} - \frac{k^*\Delta}{m} -2(1+\theta)\sqrt{\frac{\log(1/\delta)}{m}}
\EA\right.\EEQ
for some $\gamma>0$. Given $\Sht\in\symm_n$ and $\tau_\phi \in [\tau_0,\tau_1]$, the corresponding test is given by
\BEQ\label{eq:test-phi}
\ones_{\{\phi(\rho)>\tau_\phi\}}
\EEQ
The following proposition shows that if $\theta$ is high enough, then this test discriminates between $\mathcal{H}_0$ and $\mathcal{H}_1$ with probability $1-3\delta$.

\begin{proposition}\label{prop:stat-phi}
Suppose we set
\BEQ\label{eq:rho}
\Delta=4\log(9en/k^*)+4\log(1/\delta)
\quad \mbox{and} \quad
\rho=\frac{\Delta}{m}+\frac{\Delta}{\sqrt{k^*m(\Delta+4/e)}}
\EEQ
and define $\theta_\phi$ such that
\BEQ\label{eq:theta}
\theta_\phi = \left( 2\sqrt{\frac{k^*(\Delta+4/e)}{m}} + \frac{k^*(\Delta+4/e)}{m}+ 2\sqrt{\frac{\log(1/\delta)}{m}} \right) \left(1-2\sqrt{\frac{\log(1/\delta)}{m}} \right)^{-1}
\EEQ
then if $\theta> \theta_\phi$ in the Gaussian model~\eqref{eq:hyp}, the test statistic~\eqref{eq:test-phi} based on $\phi(\rho)$  discriminates between $\mathcal{H}_0$ and~$\mathcal{H}_1$ with probability $1-3\delta$.
\end{proposition}
\begin{proof}
If $ \theta_\phi$ is set as in~\eqref{eq:theta}, setting $\rho$ as in~\eqref{eq:rho} means $\Delta/m\rho <1$, we have $\tau_0 \leq \tau_1$ and propositions~\ref{prop:h0} and~\ref{prop:h1} show that~\eqref{eq:test-phi} discriminates between $\mathcal{H}_0$ and $\mathcal{H}_1$.
\end{proof}

This detection level was shown to be minimax optimal in \citep{Bert12}. This is not surprising, since the statistic $\phi(\rho)$ is simply a penalized formulation of $\lambda_\mathrm{max}^k(\cdot)$ which was shown to reach a similar detection level in \citep{Bert12}. Both $\phi(\rho)$ and $\lambda_\mathrm{max}^k(\cdot)$ are intractable however, and we will now focus on an efficiently computable statistic based on $\psi(\rho)$.

\subsection{The tractable statistic $\psi(\rho)$}
We can directly use proposition~\ref{prop:h1} to produce a lower bound on $\psi(\rho)$ and on the approximation ratio under both $\mathcal{H}_0$ and $\mathcal{H}_1$.
\begin{corollary}\label{prop:approx}
Setting $\Delta$ as in~\eqref{eq:rho}, under both $\mathcal{H}_0$ and $\mathcal{H}_1$, we have
\[
\psi(\rho) \geq 1 -\rho k^*-2\sqrt{\frac{\log(1/\delta)}{m}}
\]
with probability $1-\delta$.
\end{corollary}
\begin{proof}
We simply set $\theta=0$ in proposition~\eqref{prop:h1} and use the fact that $\phi(\rho)\leq \psi(\rho)$ by construction.
\end{proof}

We are now ready to prove the main result of this section, showing that in a high dimensional setting, the {\em tractable} statistic $\psi(\rho)$ discriminates between $\mathcal{H}_0$ and $\mathcal{H}_1$ when  $\theta\geq\theta_\psi$, where $\theta_\psi$ is comparable to $\theta_\phi$, and $\theta_\psi$ is {\em independent of $n$}. The approximation ratio in~\eqref{eq:approx-rat} is controlled by $n\rho/\psi(\rho)$ which depends on $\rho$ so we cannot explicitly minimize the detection level in $\rho$ as we did above. Instead, we will control the quality of the approximation of $\phi(\rho)$ by $\psi(\rho)$ for the value of $\rho$ used in computing $\theta_\phi$. We suppose $n=\mu m$ and $k^*=\kappa n$, where $\mu>0$ and $\kappa \in (0,1)$. Setting $\rho$ as in~\eqref{eq:rho}, we get
\[
n\rho = \mu\Delta+\frac{\mu\Delta}{\sqrt{\kappa(\Delta+4/e)}}\\
\]
with Corollary~\ref{prop:approx} implying
\[
\psi(\rho) \geq 1 -  \mu\Delta\kappa - \frac{\sqrt{{\mu\kappa}}}{\sqrt{(\Delta+4/e)}} - 2\sqrt{\frac{\log(1/\delta)}{m}}.
\]
This means that the approximation ratio in~\eqref{eq:approx-rat} is bounded below by $\beta(\mu,\kappa)$, with
\BEQ\label{eq:approx-ratio}
\beta(\mu,\kappa)=\frac{\vartheta(c)}{c} 
\quad \mbox{where} \quad
c=\frac{1 -  \mu\Delta\kappa - \frac{\sqrt{{\mu\kappa}}}{\sqrt{(\Delta+4/e)}} - 2\sqrt{\frac{\log(1/\delta)}{m}}}{{\mu\Delta}+\frac{\mu\Delta}{\sqrt{\kappa(\Delta+4/e)}}}.
\EEQ
Given $\Sht\in\symm_n$ and $\tau_\psi \in [\beta(\mu,\kappa)^{-1} \tau_0,\tau_1]$, where $\tau_0$ and $\tau_1$ are defined in~\eqref{eq:levels-phi}, the corresponding test is then
\BEQ\label{eq:test-psi}
\ones_{\{\psi(\rho)>\tau_\psi\}}
\EEQ
with $\rho$ set as in~\eqref{eq:rho}. The following proposition shows that if $\theta$ is high enough, then this test discriminates between $\mathcal{H}_0$ and $\mathcal{H}_1$ with probability $1-3\delta$.

\begin{theorem}\label{thm:test}
suppose $n=\mu m$ and $k^*=\kappa n$, where $\mu>0$ and $\kappa \in (0,1)$ are fixed and $n$ is large. Define the detection threshold $\theta_\psi$ such that
\BEQ\label{eq:theta-psi-lg}
\theta_\psi \geq \beta(\mu,\kappa)^{-1} \theta_\phi
\EEQ
where $\beta(\mu,\kappa)$ is defined in~\eqref{eq:approx-ratio} and $\theta_\phi$ is defined in~\eqref{eq:theta}, then if $\theta> \theta_\psi$ in the Gaussian model~\eqref{eq:hyp} the test statistic~\eqref{eq:test-psi} based on $\psi(\rho)$ discriminates between $\mathcal{H}_0$ and~$\mathcal{H}_1$ with probability $1-3\delta$.
\end{theorem}
\begin{proof}
Having bounded the approximation ratio $\beta(\mu,\kappa)$ defined in~\eqref{eq:approx-ratio}, the result follows from~\eqref{eq:approx-rat} and Proposition~\ref{prop:stat-phi}.
\end{proof}

In Figure~\ref{fig:beta}, we plot the level sets of $\beta(\mu,\kappa)$ for $\Delta=5$. Observe that whenever $\mu$ is small enough, $\beta(\mu,\kappa)>0$ for all values of $\kappa\in(0,1)$ and the approximation ratio converges to one as $\mu$ goes to zero. This means that the detection threshold $\theta$ of the statistic $\psi(\rho)$ remains finite when $n$ goes to infinity in the proportional regime. By contrast, the detection threshold of the MDP statistic in \citep{Bert12} blows up to infinity as soon when $k$ goes to infinity in this scenario.

\begin{figure}[h!]
\begin{center}
\psfrag{mu}[t][b]{$1/\mu$}
\psfrag{kappa}[b][t]{$\kappa$}
\psfragfig[width=.5\textwidth]{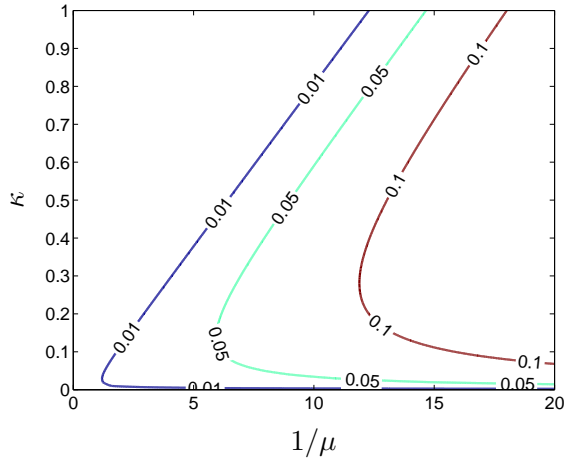}
\caption{Level sets of $\beta(\mu,\kappa)$ for $\Delta=5$.\label{fig:beta}}
\end{center}
\end{figure}

\subsection{Detection thresholds for $\psi(\rho)$ and $\lambdamax(\cdot)$}
The fact that $\psi(0)=\lambdamax(\Sigma)$ means the statistic $\psi(\rho)$ should always perform better than $\lambdamax(\cdot)$ for at least some values of $\rho$. However, we notice in Figure~\ref{fig:beta} that $\beta(\mu,\kappa)$ goes to zero as $\kappa$ goes to zero, which is a direct consequence of our choice of $\rho$ in~\eqref{eq:rho}. Our choice of $\rho$ is optimal for the statistic $\phi(\rho)$ but not for $\psi(\rho)$ and the main issue here is that we cannot explicitly maximize the difference $\tau_1-\beta(\mu,\kappa)^{-1} \tau_0$ as a function of $\rho$. On the other hand, it is easy to show that a better guess for $\rho$, when both $\kappa$ and $\mu$ are small, is to pick
\BEQ\label{eq:rho-small}
\rho=\frac{1}{n},
\EEQ
in which case, one can show that the detection threshold for $\theta_\phi$ becomes
\BEQ\label{eq:theta-smallk}
\theta_\phi = \left( \left(1+\frac{4}{e\Delta}\right)\kappa + \frac{\mu\Delta}{1-\mu\Delta}+ 2\sqrt{\frac{\log(1/\delta)}{m}} \right) \left(1-2\sqrt{\frac{\log(1/\delta)}{m}} \right)^{-1}
\EEQ
while the approximation ratio $\beta(\mu,\kappa)$ is given by $\vartheta(\psi(1/n))$ which is of order one. This means that the detection threshold for $\psi(\rho)$ is controlled by 
\[
\left(1+\frac{4}{e\Delta}\right)\kappa +\frac{\mu\Delta}{1-\mu\Delta} \simeq \left(1+\frac{4}{e\Delta}\right)\kappa + \mu\Delta
\]
when both $\kappa$ and $\mu\Delta$ are small.  On the other hand, \citep{Bena11} show that, in our regime, the statistic $\lambdamax(\cdot)$ can only distinguish between $\mathcal{H}_0$ and $\mathcal{H}_1$ when $\theta$ is larger than
\[
\sqrt{\mu}+\mu,
\]
which means that even for our suboptimal choice of $\rho$, the statistic $\psi(\rho)$ outperforms $\lambda(\cdot)$ by a factor $\Delta\sqrt{\mu}$.



\section{Algorithms}\label{s:algos}
The approximation performance we studied in the previous section comes at a price. While the semidefinite program~\eqref{eq:psi} is tractable, its complexity is significantly higher than that of the simpler relaxations derived in \citep{dasp04a}, and {\em very} significantly higher than the MDP statistic in \citep{Bert12} for example. \cite{dAsp08b} derived greedy algorithms to produce good upper bounds on $\psi(\rho)$ from approximate dual solutions to problem~\eqref{eq:psi}, but there are of course no guarantees on the quality of their output. In this section, we describe a simple algorithm to compute $\psi(\rho)$, with comparatively low storage and iteration costs.

Recall from \citep{dAsp08b} that the dual to problem~\eqref{eq:psi} is written
\BEQ\label{eq:psi-dual}
\BA{ll}
\mbox{minimize} & \lambdamax\left(\sum_{i=1}^n Y_i\right)\\
\mbox{subject to} & Y_i \succeq a_ia_i^T -\rho \idm\\
& Y_i \succeq 0,\quad i=1,\ldots,n
\EA\EEQ
in the variables $Y_i\in\symm_n$, where $\rho>0$ and $a_i\in\reals^n$ are defined in Section~\ref{s:relax}. We first show how to regularize this problem, then discuss how to solve the regularized instance using a Frank-Wolfe type algorithm.

\subsection{Smoothing}
Problem~\eqref{eq:psi-dual} is not smooth but we can write 
\[
\lambdamax(Y)=\max_{\Tr(X)=1,\,X\succeq 0}~ \Tr(YX)
\]
and as in \citep{BenT04}, a natural way to regularize $\lambdamax(\cdot)$ is to add a (strongly convex) matrix entropy penalty to this variational formulation. We also  add an explicit lower bound on the eigenvalues of $X$ to ensure that the gradient matrix $X$ is invertible and well conditioned. We summarize this in the next lemma.

\begin{lemma}\label{lem:smoothing}
Let $\epsilon>0$, the function
\BEQ\label{eq:f-smooth}
f(Y) \triangleq~ \max_{\substack{\Tr(X)=1,\\X\succeq (\epsilon/n) \idm}}~ \Tr(YX) - \frac{\epsilon}{\log n} (\Tr(X\log(X)) + \log n)
\EEQ
has a Lipschitz continuous gradient with constant
\[
L_f \leq \frac{\log n}{\epsilon}
\]
with respect to the trace norm and satisfies
\[
(1-\epsilon)\lambdamax(Y) -\epsilon \leq f(Y) \leq \lambdamax(Y).
\] 
for all $Y\in\symm_n$.
\end{lemma}
\begin{proof}
In the spectahedron setting, we know from \citep{BenT04} that the matrix entropy function
\[
d(X)=\Tr(X\log(X)) + \log n
\]
is strongly convex with parameter $1/2$ with respect to the trace norm (the dual of the spectral norm) and satisfies 
\[
\max_{\Tr(X)=1,\,X\succeq 0}~d(X) \leq \log n.
\]
Then, \citep[Thm.\,1]{Nest03} shows that $\nabla f(Y)$ is Lipschitz continuous with constant $\log n/\epsilon$ with respect to the trace norm. By construction, we have $f(Y) \leq \lambdamax(Y)$ and 
\[
(1-\epsilon) \lambdamax(Y) - \epsilon \leq \max_{\substack{\Tr(X)=1,\\X\succeq (\epsilon/n) \idm}}~ \Tr(YX) -\epsilon \leq f(Y)
\]
hence the desired result.
\end{proof}

\begin{algorithm}[hb]
\caption{Frank-Wolfe algorithm for computing $\psi(\rho)$.} 
\label{alg:fw} 
\begin{algorithmic} [1]
\REQUIRE $\rho>0$ and a feasible starting point $Z_0$. 
\FOR{$k=1$ to $N_{max}$} 
\STATE Compute $X=\nabla f(Z)$, together with $X^{-1}$ and $X^{1/2}$.
\STATE Solve the $n$ subproblems
\BEQ\label{eq:aff-fw}
\BA{ll}
\mbox{minimize} & \Tr(Y_iX)\\
\mbox{subject to} & Y_i \succeq a_ia_i^T -\rho \idm\\
& Y_i \succeq 0,
\EA\EEQ
in the variables $Y_i\in\symm_n$ for $i=1,\ldots,n$.
\STATE Compute $W=\sum_{i=1}^n Y_i$.
\STATE Update the current point, with
\[
Z_k=\left(1-\frac{2}{k+2}\right) Z_{k-1} + \frac{2}{k+2} W,
\]
\ENDFOR
\ENSURE A matrix $Z\in\symm_n$.
\end{algorithmic} 
\end{algorithm} 

We can thus form a smooth approximation of problem~\eqref{eq:psi-dual}, written
\BEQ\label{eq:f-prob}
\BA{ll}
\mbox{minimize} & f\left(\sum_{i=1}^n Y_i\right)\\
\mbox{subject to} & Y_i \succeq a_ia_i^T -\rho \idm\\
& Y_i \succeq 0,\quad i=1,\ldots,n
\EA\EEQ
in the variables $Y_i\in\symm_n$, where $f$ is the  smooth approximation of the function $\lambdamax(\cdot)$ defined in~\eqref{eq:f-smooth} and solve this problem using Algorithm~\ref{alg:fw}. 

\subsection{Complexity} We first show how to efficiently compute both $f(Z)$ and  $\nabla f(Z)$. 
\begin{lemma}
Assume $Z=V\diag(y)V^T$ and suppose $\lambda\in\reals$ solves
\BEQ\label{eq:dual-grad}
\min_\lambda ~ \sum_{i=1}^n \min \left\{\frac{\epsilon(y_i+\lambda)}{n}- \frac{\beta\epsilon}{n}\log\frac{\epsilon}{n}, \beta e^{\frac{y_i+\lambda}{\beta}-1}
\right\}
\EEQ
then $Y=V\diag(x)V^T$ solves the maximization problem in~\eqref{eq:f-smooth}, where
\[
x_i=\max\left\{\frac{\epsilon}{n},e^{\frac{y_i+\lambda}{\beta}-1}\right\}, \quad i=1,\ldots,n
\]
and we have $\nabla f(Z)=Y$.
\end{lemma}
\begin{proof}
The function $f(Z)$ is spectral so solving~\eqref{eq:f-smooth} is equivalent to solving
\[
\max_{\substack{\ones^Tx=1,\\ x \geq (\epsilon/n)}}~ y^Tx - \frac{\epsilon}{\log n} \sum_{i=1}^n x_i\log x_i
\]
whose dual is 
\[
\min_\lambda ~\max_{x \geq \epsilon/n} ~ (y-\lambda \ones)^Tx - \frac{\epsilon}{\log n} \sum_{i=1}^n x_i\log x_i - \lambda
\]
which is equivalent to~\eqref{eq:dual-grad}. Now, the fact that $f(Z)$ is a maximum of affine functions of $Z$ shows that $\nabla f(Z)=Y$.
\end{proof}

Besides computing the gradient, the main cost at each iteration is the problem of solving the $n$ subproblems~\eqref{eq:aff-fw}. We will see that when $\nabla f(Z)$ is positive definite, then this can be done in closed form, with complexity $O(n^2\log n)$. Furthermore, the matrices $Y_i$ do not need to be stored, only their sum is required.
\begin{lemma}\label{lem:aff-fw}
Given $X\in\symm_n$ such that $X\succ 0$, together with $B_i=a_ia_i^T-\rho \idm$ for some $\rho>0$, we have
\[
\min_{\substack{Y\succeq B_i\\Y\succeq 0}} \Tr(XY) ~ = ~\Tr(X^{1/2}B_iX^{1/2})_+
\]
and the optimal solution has rank one and is given by
\[\left\{\BA{ll}
Y=X^{-1/2}vv^TX^{1/2}B_i, & \mbox{if }\rho<\|a_i\|_2^2,\\
0 & \mbox{otherwise,}
\EA\right.
\]
where $v$ is the leading eigenvector of the matrix $X^{1/2}B_iX^{1/2}$.
\end{lemma}
\begin{proof}
Recall from \cite{dAsp08b} that
\[\BA{ll}
\Tr(X^{1/2}B_iX^{1/2})_+&=\dsp\max_{\{0\preceq P \preceq X\}} \Tr(PB_i)\\
&=\dsp\min_{\{Y\succeq B,~Y\succeq 0\}} \Tr(YX),
\EA\]
and a solution to the dual of problem~\eqref{eq:aff-fw} can be obtained from the solution to
\[
\max_{\{0\preceq Q \preceq \idm\}} \Tr(QX^{1/2}B_iX^{1/2}).
\]
By Sylvester's theorem, when $\rho<\|a_i\|_2^2$, the matrix $X^{1/2}B_iX^{1/2}$ has exactly one nonnegative eigenvalue, so the optimal solution to this last problem is $Q=vv^T$ where $v$ is the leading eigenvector of the matrix $X^{1/2}B_iX^{1/2}$. This means that the optimal dual solution is $P=X^{1/2}vv^TX^{1/2}$. Finally, the KKT optimality conditions impose $XY=PB_i$, which together with $X\succ 0$ yields the desired result.
\end{proof}

The last lemma shows that solving problem~\eqref{eq:aff-fw} requires the following steps. Assume $X^{-1}$ and $X^{1/2}$ have been precomputed, we first form the matrix $X^{1/2}B_iX^{1/2}$ at cost $O(n^2)$. We then compute its leading eigenvector at cost $O(n^2\log n)$ and form the rank one matrix $P=X^{1/2}vv^TX^{1/2}$ at cost $O(n^2)$. Because $P$ is rank one, computing $X^{-1}PB_i$ also costs $O(n^2)$. This means that the total cost of solving problem~\eqref{eq:aff-fw} is bounded by $O(n^2\log n)$. Furthermore, without loss of generality, we can restrict the matrices $Y_i$ to have norm less than
\[
B=\left\|\sum_{i=1}^n B_i^+\right\|_2
\]
which means we can assume the feasible set of problem~\eqref{eq:f-prob} is compact.

\begin{proposition}\label{prop:fw-cost}
Assume $A\in\reals^{n\times n}$ and $\rho>0$, Algorithm~\ref{alg:fw} will produce and $\epsilon$ solution to problem~\eqref{eq:f-prob} in
\[
O\left(\frac{D^2 \log n}{\epsilon^2}\right)
\]
iterations, where $D$ is the trace norm diameter of the feasible set
\[
D=\diam \left\{\sum_{i=1}^nY_i : Y_i\succeq B,\,Y_i\succeq 0,\,\|Y_i\|_2\leq B\right\}.
\]
Each iteration has complexity $O(n^3\log n)$ and storage cost $O(n^2)$.
\end{proposition}
\begin{proof}
We use the complexity bounds in \citep{Fran56,Clar10,Jagg11} for example and Lemma~\eqref{lem:smoothing} to control the curvature of $f$.
\end{proof}


\section{Numerical Experiments}\label{s:numexp}
We test the detection procedure based on $\psi(\rho)$ described in~\eqref{eq:test-psi}. We generate $3000$ experiments, where $m$ points $x_i\in\reals^n$ are sampled under both hypotheses, with
\[\left\{\BA{rl}
\mathcal{H}_0:&x \sim\mathcal{N}\left(0,\idm_n\right)\\
\mathcal{H}_1: &x \sim\mathcal{N}\left(0,\idm_n+\theta vv^T\right)
\EA\right.\]
as in~\eqref{eq:hyp}. In each experiment, we pick the leading dimension $n=100$, the number of samples $m=50$ and the cardinality $k=20$. We set $\theta=3$, $v_i=1/\sqrt{k}$ when $i\in[1,k]$ and zero otherwise. In Figure~\ref{fig:test-psi} we plot the distributions of the test statistic $\psi(\rho)$ defined in~\eqref{eq:test-psi}, the MDP statistic in \citep{Bert12}, the $\lambdamax(\cdot)$ statistic and the diagonal statistic in \citep{Amin08}. As in \citep{Bert12}, we observe that all the test statistics perform very similarly, except for the diagonal test.

\begin{figure}[!ht]
\begin{center}
\begin{tabular}{cc}
\psfrag{x}[t][b]{$\psi(\rho)$}
\psfrag{h0}[c][c]{$\mathcal{H}_0$}
\psfrag{h1}[c][c]{$\mathcal{H}_1$}
\psfragfig[width=.4\textwidth]{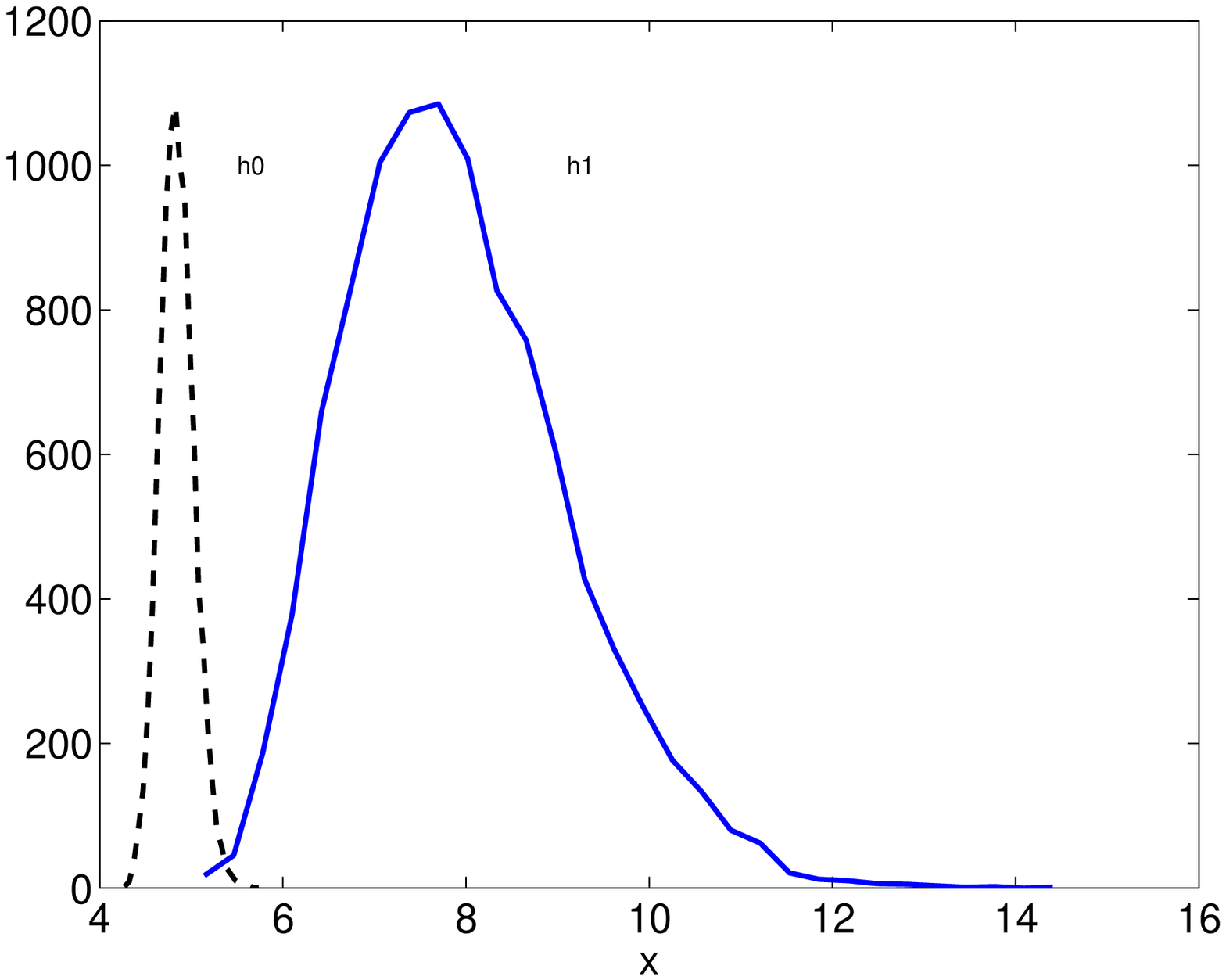}& 
\psfrag{x}[t][b]{$MDP_k$}
\psfrag{h0}[l][c]{$\mathcal{H}_0$}
\psfrag{h1}[l][c]{$\mathcal{H}_1$}
\psfragfig[width=.4\textwidth]{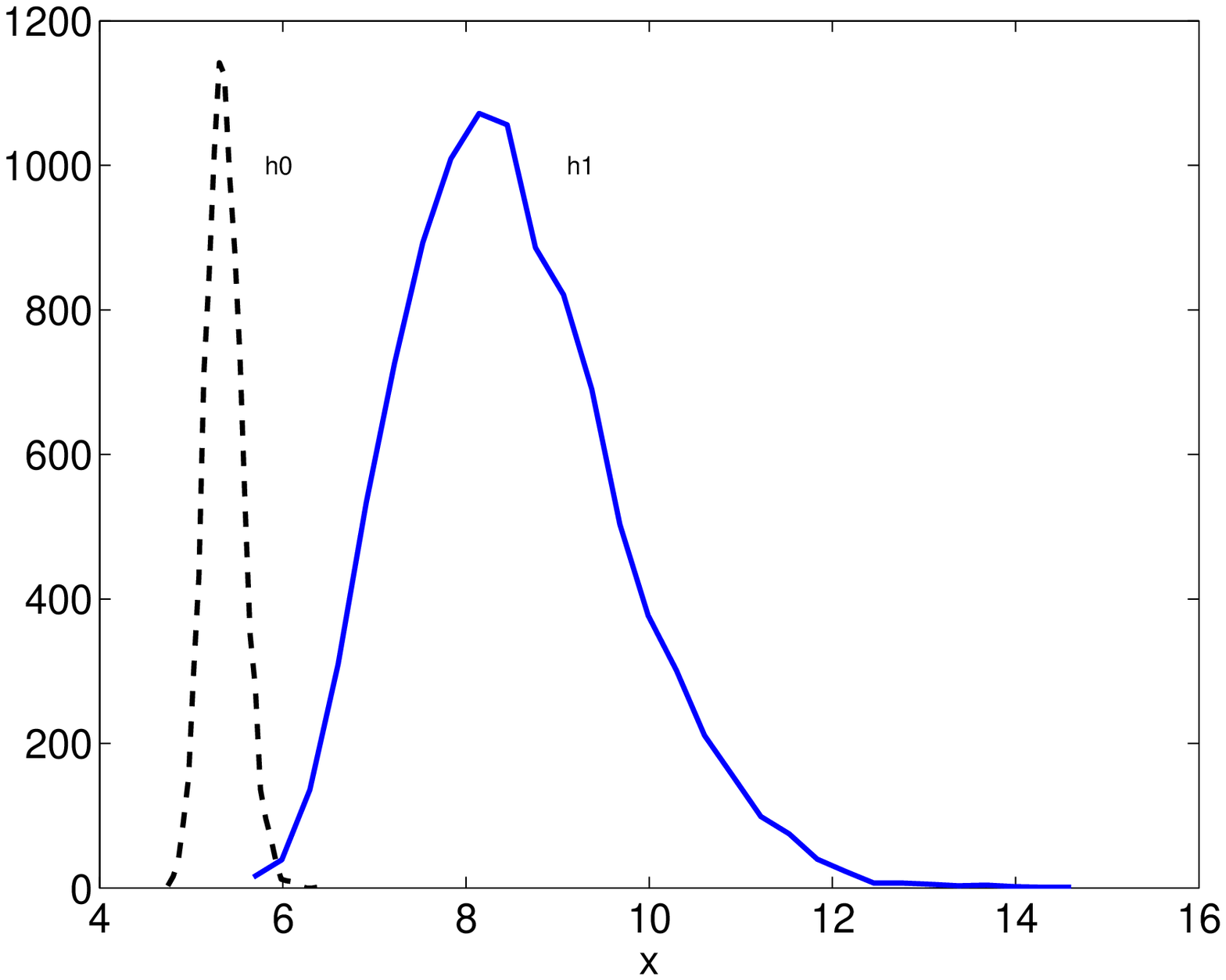}\\
\psfrag{x}[t][b]{$\lambdamax(\hat \Sigma)$}
\psfrag{h0}[l][c]{$\mathcal{H}_0$}
\psfrag{h1}[l][c]{$\mathcal{H}_1$}
\psfragfig[width=.4\textwidth]{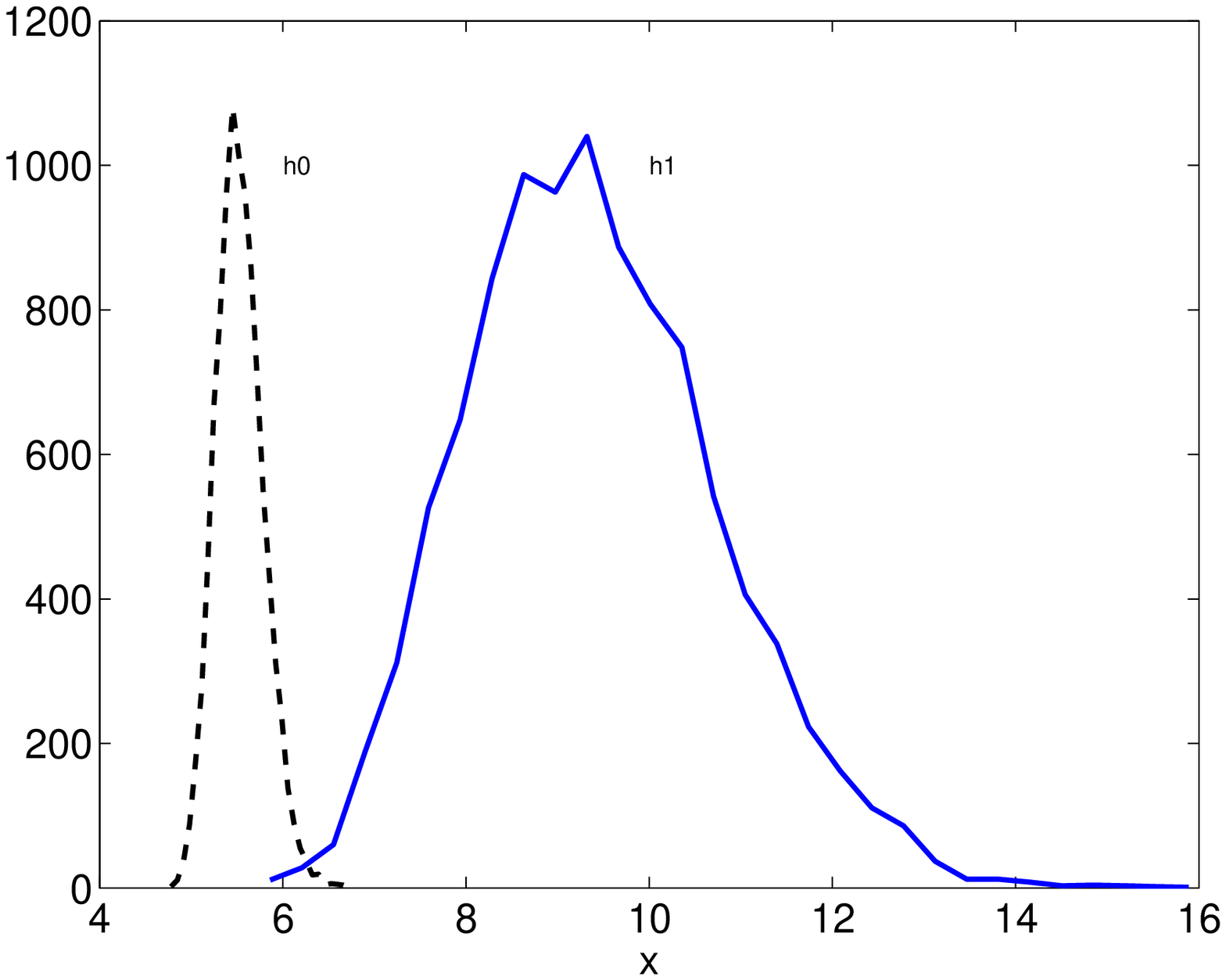}& 
\psfrag{x}[t][b]{$\max(\diag(\hat \Sigma))$}
\psfrag{h0}[c][c]{$\mathcal{H}_0$}
\psfrag{h1}[c][c]{$\mathcal{H}_1$}
\psfragfig[width=.4\textwidth]{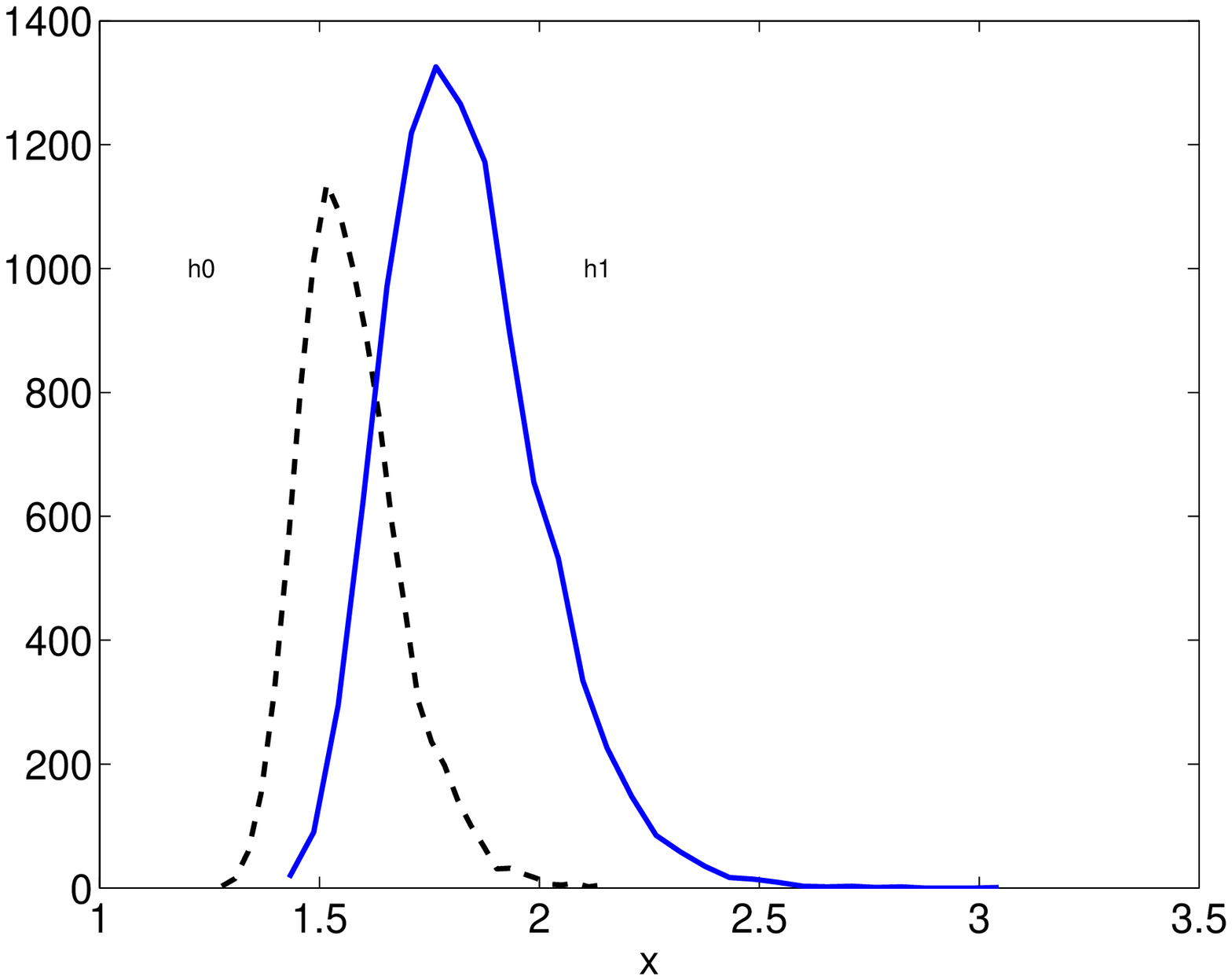}
\end{tabular}
\caption{Distribution of test statistic $\psi(\rho)$ {\em (top left)}, the $MDP_k$ statistic in \citep{Bert12} {\em (top right)}, the $\lambdamax(\cdot)$ statistic {\em (bottom left)} and the diagonal statistic from \citep{Amin08} {\em (bottom right)} under both $\mathcal{H}_0$ and $\mathcal{H}_1$. All experiments are performed on random Gaussian matrices with ambient dimension $n=100$, $m=50$ samples and cardinality for $v$ under $\mathcal{H}_1$ set to $k=20$. \label{fig:test-psi}}
\end{center}
\end{figure}

\section*{Acknowledgments}  AA would like to acknowledge partial support from NSF grants SES-0835550 (CDI), CMMI-0844795 (CAREER), CMMI-0968842, a starting grant from the European Research Council (project SIPA), a Peek junior faculty fellowship, a Howard B. Wentz Jr. award and a gift from Google. The authors would like to thank Philippe Rigollet and Quentin Berthet for very constructive discussions and introducing them to the detection problem in Section~\ref{s:detect}.

\small{\bibliographystyle{plainnat}
\bibsep 1ex
\bibliography{MainPerso}}
\end{document}